\documentclass[11pt]{amsart}

\usepackage{amssymb,amsmath,amsfonts,amsthm,amscd}
\usepackage{graphics}
\usepackage[dvips]{color}

\input xy
\xyoption{all}

\newcommand{\CC}{\mathbb C}
\newcommand{\RR}{\mathbb R}
\newcommand{\maA}{\mathcal{A}}
\newcommand{\maH}{\mathcal{H}}
\newcommand{\maK}{\mathcal{K}}
\newcommand{\maB}{\mathcal{B}}
\newcommand{\maC}{\mathcal{C}}
\newcommand{\maD}{\mathcal{D}}

\newcommand{\gp}{\Gamma}

\newcommand{\mbR}{\mathbb{R}}
\newcommand{\ZZ}{\mathbb{Z}}

\newcommand{\maM}{\mathcal{M}}
\newcommand{\maF}{\mathcal{F}}

\newcommand{\maU}{\mathcal{U}}
\newcommand{\maL}{\mathcal{L}}
\newcommand{\maI}{\mathcal{I}}

\newcommand{\maV}{\mathcal{V}}

\newcommand{\pdoM}{\Psi^{\infty}(M)/\Psi^{-\infty}(M)}
\newcommand{\smooth}{\mathcal{C}^{\infty}}
\newcommand{\graded}{\mathcal{G}\textrm{r}}

\newcommand{\bx}{{\bf x}}
\newcommand{\bxi}{\mbox{\boldmath$\xi$}}
\newcommand{\bfeta}{\mbox{\boldmath$\eta$}s}
\newcommand{\ip}[1]{\langle #1 \rangle}
\newcommand{\bigip}[1]{\left\langle #1 \right\rangle}

\theoremstyle{definition}
\newtheorem{definition}{Definition}[section]

\theoremstyle{plain}
\newtheorem{theo}[definition]{Theorem}
\newtheorem{prop}[definition]{Proposition}
\newtheorem{lem}[definition]{Lemma}
\newtheorem{cor}[definition]{Corollary}

\theoremstyle{remark}
\newtheorem{remark}[definition]{Remark}

\begin{document}

\textcolor{red}{
 \title[Equivariant homology for  pseudodifferential operators] {Equivariant homology for  pseudodifferential operators} 
}

\author[S Dave]{Shantanu Dave}
\address{University of Vienna, Austria }
\email{shantanu.dave@unvie.ac.at}
\thanks{Supported by FWF grant Y237-N13 of the Austrian Science Fund.}

\begin{abstract}We compute the cyclic homology for the cross-product algebra
$\maA(M)\rtimes\gp$ of the algebra of complete symbols on a compact manifold $M$ with action of  a  finite group $\gp$. A
spectral sequence argument shows that these groups can be identified using
deRham cohomology of the fixed point manifolds $S^*M^g$. In the process we obtain new results
about the homologies of general cross-product algebras and provide explicit identification of the homologies
for  $\smooth(M)\rtimes \gp$. 
\end{abstract}

\maketitle

\section{Introduction}\label{sec_intro}

On a closed manifold  $M$ the (classical)
pseudodifferential operators form an algebra $\Psi^{\infty}(M)$. The space of smoothing
operators $\Psi^{\!-\infty}(M)$ is then an ideal and  the quotient
$\maA(M):=\pdoM$ is called the algebra of complete symbols.  Let $\gp$ be a finite group acting on $M$ by
diffeomorphisms. Then by push-foreword of operators $\gp$ acts on
$\Psi^{\infty}(M)$ and on $\maA(M)$, namely  if $\maD$ be  a pseudodifferential operator and
$g\in\gp$, then
\[g.\maD (f):=g\maD(g^{\!-\!1}f)\qquad  \forall f \in \smooth(M).\] 
In this paper we compute the
Hochschild and cyclic homology groups of the cross-product algebra $\maA(M) \rtimes \Gamma$.

 Results on the cyclic homology for algebras of complete symbols over compact manifolds were  obtained  in \cite{Br-Gz,Wodzicki}. In particular, these homology calculations  recover the noncommutative residue of {Guillemin}\cite{Guillemin} and {Wodzicki}\cite{Wodzicki}. 
In a similar way among other things our culations of these homology groups tells us exactly how many
linearly independent equivariant traces to expect on the algebra
$\maA(M) \rtimes \Gamma$. These traces are computed in \cite{Dave} where
they are 
considered as equivariant generalization of the noncommutative residue 
Wodzicki\cite{Wodzicki} and Guillemin \cite{Guillemin}.

We shall begin with the motivation for our calculationss.

As  noticed in \cite{Dave} that certain traces on the crossed product algebra  $\maA(M)\rtimes \Gamma$ can be considered as equivariant versions of the noncommutative residue. To a generator $Ag$ in $\maA(M)\rtimes \Gamma$ and an equivariant oder $1$ positive elliptic operator $D$ wone associates the  zeta function
\[\zeta_{Ag,\maD}(z):= \operatorname{Tr}(D^{-z}Ag).\]
By means of stationary phase analysis  near the fixed points of the diffeomorphism  $g$, a meromorphic extension of these zeta functions to whole of $\CC$ with some simple poles can be shown. Then for a fiexed conjugacy class $\ip{\gamma}$ a trace on $\maA(M)\rtimes \Gamma$ is obtained by
\[\operatorname{Tr}_{\ip{\gp}}(\sum_{g\in\gamma}A_gg):=\underset{z=0}{res}(\sum_{g\in\ip{\gp}}\zeta_{A_g,\maD}(z).\]

Analogous to  Guillemin \cite{Guillemin} the equivariant  traces can be used to obtain an equivariant Weyl's formula \cite{}, namely if $\\Gamma$ acts faithfully  and $\pi$ is an irreducibel representation of $\Gamma$ then for an invariant operator $\maD$
\[  N_{\pi,D}(\lambda):=\sum_{\lambda_i<\lambda}\
    \textrm{``multiplicity of}\ \pi\ \textrm{in}\ V_i.\textrm{''}\simeq \frac{C}{dim~\pi}\lambda^{\frac{dim~M}{order(\maD)}}
\]

Other equivariant results such as an equivariant  Connes trace formula as well as extensions of the logrithmic symbols  based on a $2$  cocycle in $H^2(S_{log}(M)\rtimes \Gamma)$ as in  \cite{KV} can also be obtained from the above mentioned traces on $\maA(M)\rtimes \Gamma$.

Here we are interesed in knowing the higher versions of these equivariant noncommutative residues.

Our result is as follows. Let
$\Gamma_{\gamma} :=\{g \in \Gamma, g\gamma = \gamma g\}$ be the
centralizer of $g$ in $\Gamma$.  Let
$k_{\gamma}=\dim(T^*M^{\gamma})$. Then
\[
HH_k(\maA(M)\rtimes \Gamma)=\sum_{\langle\gamma\rangle}
H^{k_{\gamma}-k}(S^*M^{\gamma}\times S^1)^{\Gamma_{\gamma}}.
\]
Here the sum is taken over a set of representatives of the conjugacy
classes. Also
\[HC_k(\maA(M)\rtimes \Gamma) = \sum_{j\ge 0}
HH_{k-2j}(\maA(M)\rtimes \Gamma).\]

Our determination of these homology groups extend the results of
\cite{Br-Gz}, using also techniques from \cite{Br-Ni, Co}.  Interestingly, there are some
qualitatively new phenomena arising at the nontrivial conjugacy
classes that are not expected from the non-equivariant case. 

The cross-product algebra $\maB:=\maA(M)\rtimes \Gamma$ has a
natural filtration that comes from the order of the operators on
$\maA(M)$. We use the spectral sequence associated to this
filtration in our homological computation. The first hurdle here is
that the associated graded algebra $Gr(\maB)=Gr(\maA(M))\rtimes
\Gamma$ is noncommutative, unlike $Gr(\maA(M))$ which is
commutative. Nevertheless, this algebra is the cross-product of a
commutative algebra by a finite group, and as such it preserves many
features of commutativity. In particular, its Hochschild
homology has a description using differential forms on the fixed
point sets of the elements of the group \cite{BaumConnes}. The differentials in the
spectral sequence turn out to preserve this structure. The action of
the first relevant differential, $d_2$, is similar to the one in the
case without group action \cite{Br-Gz}, albeit technically different. Here one
can exploit the  structure of certain symplectic submanifolds of the the
cotangent bundle. Moreover, the residue trace associated to each
conjugacy class of $\Gamma$ (provided that that conjugacy class has
a nonempty, connected fixed point set) will no longer have the
property of being localized to a singly homogeneous component of the
symbol in any coordinate neighborhood. This is in contrast with the
usual case when there is no group action, the residue
trace is localized on the component of homogeneity $-n$ of the
complete symbol. Moreover, the study of the equivariant residue
traces requires a nontrivial use of the stationary action principle.

As mentioned above, we need to know as explicitly as possible the
Hochschild homology groups of $\smooth(S^*M)\rtimes \Gamma$. In the
process of these computations, several new results on these groups are obtained. These
results fit into the general philosophy of noncommutative geometry
that Hochschild homology is the analogue of smooth forms on a
compact manifold and (periodic) cyclic homology is the analogue of
  deRham cohomology for manifolds. These results are consistent with
a famous theorem of Connes that computes the Hochschild and cyclic
homology groups of $\smooth(M)$, the algebra of smooth functions on
$M$. They are also consistent with the results of Baum and Connes
\cite{BaumConnes} on the homology of cross-products by proper
actions.

\section{Hochschild and cyclic homology for cross-products}
We recall the definitions for Hochschild and cyclic homology of
algebras. We describe the properties needed, and set up our notation
for the subsequent sections. A reference for most of the results cited
here is Loday\cite{Loday}.  Unless otherwise stated all algebras in this section
are over the complex numbers and shall be unital.

The Hochschild homology of an unital algebra $\maA$, denoted by  $HH_*(A)$ is the
homology of the complex $\maH_*(A):=(\maA^{\otimes n+1},b)$ where the differential $b$ is given by
\begin{multline*}
b(a_o\otimes a_1\otimes a_2 \otimes\ldots \otimes a_n):= \sum_{i=0}^{n-1} (\!-\!1)^i a_0\otimes a_1\otimes a_2
\otimes\ldots \otimes a_ia_{i+1} \ldots a_n\\ +(\!-1\!)^n a_na_0\otimes a_1\otimes a_2
\otimes\ldots \otimes a_{n-1}.
\end{multline*}

Let $\maA$ be an algebra and $\gp$ be a finite group acting on it by
$\pi : \gp \rightarrow Auto(\maA)$. For most purposes, $\maA$ will be a
locally convex topological algebra with jointly continuous product. We
define the cross-product algebra $\maB=\maA\rtimes \gp$ as the algebra
generated by elements of the form $\{a_gg|a_g\in
\maA\,,g\in\gp\}$ with the product given by

\[a_gg \cdot b_hh=a_gg(b_h)gh.\]
The Hochschild homology of $\maB$ admits a natural decomposition.
For every conjugacy class $\langle \gamma \rangle $ of the group
$\gp$, there is a  subcomplex of $\maH_*(\maB)$ given by
\[\maH_*(\maA)_{\gamma}= \{(a_{g_0}g_0\otimes a_{g_1}g_1\otimes
\ldots \otimes a_{g_n}g_n)|\; g_1.g_2\ldots g_n.g_0 \in\langle \gamma
\rangle\},\] which yields  the decomposition:
\[ \maH_*(\maB)=\bigoplus_{\langle \gamma
\rangle}\maH_*(\maA)_{\gamma}.\]

We shall also use the notation $L_*(\maA,\gp,\gamma)$ for
$H_*(\maA)_{\gamma}$ if we need to specify the group $\gp$
explicitly.

Our aim is to first identify each of the homology of a conjugecy components $\maH_*(\maA)_{\gamma}$ with that of a certain twisted
Hochschild complex.

\subsection{Twisted Hochschild Complex for Commutative
Algebras}\label{sub-THCC}

We consider a commutative algebra $\maA$. Let $h$ be an
automorphism of $\maA$. Let as usual $\maA^e=\maA\otimes \maA$. Consider the
$\maA^e$ module $\maA_h$ with the same linear structure as
$\maA$, but the module structure given by
\[(a \otimes b)\cdot c= ac \cdot h(b).\]

Furthermore let us consider the complex $\maC_*(\maA)_h= \bigoplus
\maA^{\otimes n+1},b_h$, where the twisted Hoschild differential $b_h$ is defined as
\begin{equation}\label{twisted}
b_h(a_0 \otimes a_1 \otimes \ldots \otimes a_n)
= (a_0h(a_1)\otimes \ldots \otimes a_n)+ 
\sum_{i=1}^n(-1)^i (a_0 \otimes \ldots \otimes a_ia_{i+1} 
\otimes \ldots \otimes a_n).
\end{equation}

In case of a group action on a manifold $M$ we shall be able to very easily identify the homology of twisted complex $\maC_*(\smooth(M))_h$ with the differential forms on the fixed point manifold $M^h$. At he same time the result in this cetion show that twisted homology gives the same homology as the conjucacy component  $\maH_*(\smooth(M)\rtimes \Gamma)_{\ip{h}}$.

Then we have the following result.

\begin{lem}\label{Tor}\ We have
$\maC_*(\maA)_h \simeq Tor^{\maA^e}_*(\maA,\maA_h).$
\end{lem}

\begin{proof}
Let $\maH'_*(\maA)$ denote the bar resolution of $\maA$. That is,
$\maH'_*(\maA)=(\maA^{\otimes n+1},b')$ is the standard projective
resolution of $\maA$ by $\maA^e$ modules. Consider the map
\[\psi : \maH_n(\maA) \otimes_{\maA^e}  \maA_h 
\longrightarrow \maC_n(\maA)_h\]
\[\psi(a_0\otimes a_1 \otimes \ldots \otimes a_n)\otimes a:\rightarrow
(a_nah(a_0)\otimes a_1 \otimes \ldots \otimes a_{n-1}).\]

Then we check that
 $b_h \circ \psi = \psi \circ b' \otimes 1$,
which means that $\psi$ is a morphism of complexes.
The result follows from the fact that $\psi$ is an isomorphism
and the definition of the $Tor$ groups.
\end{proof}

Abstracty the above result  can tell us that the twisted Hoschild complex $\maC_*(\maA)_h$  has the same homology as the conjugacy  component $\maH_*(\maA)_{\ip{h}}$ but we are in the lookout for a concrete qusisiomorphism.

 To this end we  start with another well known acyclic moddel. For any finite group $G$, let $\beta G$, $(\beta G)_n
=\mathbb{C}[\underbrace{G\times G\times \dots \times G}_{\text{ n
times}}]$, be the complex endowed with the differential
\begin{eqnarray*}
\overline{d}(g_0,g_1,\dots
,g_n)=\sum_{i=0}^n(-1)^i(g_0,g_1,\dots,\hat{g_i},\dots,g_n),
\end{eqnarray*}
where $\hat{g_i} $ as usual means that the entry is omitted. This
differential also comes from a simplicial object structure on $\beta
G$. It is well known that
$$H_q(\beta G)=\left\{\begin{array}{c}
    \CC \hskip 0.5in \text{if } q=0\\
    0, \hskip 0.5in q>0
\end{array} \right.$$
and hence $\beta G$ is a free resolution for the trivial $G$ module
$\mathbb{C}$.

 For any subgroup $G \subseteq \gp$, we define the
complex
\[\tilde{L}_n(\maA,G,h):=(\maC_*(\maA)_h)_n \otimes \beta G_n\]
with induced simplicial
structure and differential given by
\begin{multline*}
 b_h\otimes \overline{d}(a_0g_0,a_1g_1, \dots ,a_ng_n)=(a_0 h(a_1)g_1,a_2g_2,\dots,a_ng_n)+\\
\sum_{i=1}^{n-1}(-1)^i(a_0g_0,\dots,a_ia_{i+1}g_{i+1}, \dots ,a_ng_n) +\\
(-1)^n(a_na_0g_0,a_1g_1 \dots ,a_{n-1}g_{n-1}).
\end{multline*}
 As one would expect this  $\tilde{L}_*(\maA,G,h)$ complex to be a simply connected cover of $\maH_*(\maA)_{\ip{h}}$.( See \cite{Nistor1} for more detailed presentation.)
 
An application  of the Eilenburg-zibler isomorphism  gives us the following result.
\begin{lem}
We have $\tilde{L}(\maA,G,h)$ is quasi-isomorphic to  $\maC_*(\maA)_h.$
\end{lem}

\begin{proof}
By definition, $\tilde{L}(\maA,G,h)_n \simeq \beta G_n
\otimes(\maC_*(\maA)_h)_n$. Thus by the K\"unneth formula and the
Eilenberg-Zibler theorem, we have the following diagram:
\begin{center}
$ \xymatrix{ {\tilde{L}}(\maA,G,h) \ar[rr]^{\pi_G} \ar[dr]^f & &
{\maC_*(\maA)_h }\ar@<1ex>[ll]^{\pi^G}\\ & \beta G \otimes
{\maC_*(\maA)_h} \ar@<1ex>[ul]^g \ar[ur]^{\pi} },$
\end{center}
where we have denoted
$$\pi_G (a_0g_0,a_1g_1, \ldots ,a_ng_n) :=(a_0,a_1,\ldots ,a_n)$$
and similarly,
$$\pi^G (a_0,a_1,\ldots ,a_n):=(a_0e,a_1e,\ldots,a_ne),$$ and also the
maps $f$ and $g$ are the maps for the Eilenberg-Zibler
quasi-isomorphism, and $\pi$ is the projection on the first
component. 
(The Eilenberg-Zibler isomorphism applies because the complex
$\tilde{L}(\maA,G,h)$ is obtained from the product of two simplicial
objects.) That is, $$\pi((\beta G)_l
\otimes \maC_k(\maA))=\left\{\begin{array}{cl}
0 \hskip 0.5in &\text{if }l \neq 0\\
\maC_k(\maA) \hskip 0.5in
&\text{if }l=0.
\end{array} \right.$$

Since $H_q(\beta G)=0$ for $q>0$, $\pi$ turns out to be a
quasi-isomorphism. Thus $\pi_G$ and $\pi^G$ are quasi-isomorphisms.
\end{proof}

  Let $\maA$ be a commutative  unitial  algebra with an action of a group $\Gamma$ .
Let $h\in\ip{\gamma}$ be an element of the conjugacy class $\ip{\gamma}$ and
$\maC_*(A)_h$ be the corresponding  twisted complex under the action of $h$,
then we can make the following identifications:

\begin{prop}\label{cp} The conjugacy component of the Hochschild homology
$\maH_*(\maA)_{h}$ is quasi-isomorphic to $\maC_* {(\maA)_h}^{\gp_h}$.

And the chain map $G:\maH_*(\maA)_{\ip{\gamma}} \longrightarrow
\maC_*(\maA)_h^{\Gamma_h}$ is given by the explicit formula
\begin{multline}\label{quasi-isomorphism}
	G(b_0h_0,\, b_1h_1,\ldots,\,
	b_nh_n)=\\
	\frac{1}{|\gp_h|}\sum_{g_0 \in
	\gp_h}(hg_0h_0^{-1}(b_0),\, g_0(b_1),\ldots,\, g_0h_1\ldots
	h_{n-1}(b_n)).
\end{multline}
\end{prop}

\begin{proof}

There is a \textit{covering map } $\alpha :\tilde{L}(\maA,\gp,h)
\longrightarrow \maH_*(\maA)_{\gamma}$ given by \
\begin{multline*}
\alpha(a_0g_0,\, a_1g_1,\, \ldots ,\, a_ng_n) =
(g_n^{-1} a_0 hg_0,\, g_0^{-1} a_1 g_1,\, \ldots,\,
g_{n-1}^{-1} a_n g_n) \\ =
(g_n^{-1}(a_0)g_n^{-1}hg_0,\, g_0^{-1}(a_1)g_0^{-1}g_1,\, \ldots,\,
g_{n-1}^{-1}(a_n)g_{n-1}^{-1}g_n),
\end{multline*}
which is a chain map. In fact, $\alpha$ is a morphism of simplicial
objects and $\gp_h$ equivariant. We would like to lift this map $\alpha$ from $\maH_*(\maA)_{\ip{h}}$ to $\tilde{L}(\maA,\gp,h)^{\gp_h}$.

  The map $\alpha$ above also restricts to a chain map on the qusi-isomorphic complex $\tilde{L}(\maA,\Gamma_h,h)$  which we denote by 
$\alpha_{|\gp_h}$. An explicit lifting is easy to construct for
$\alpha_{\gp_h}$ as follows:

Define, for any $g_0 \in \gp_h$ a linear map
$T_{g_0}:L(\maA,\gp_h,h) \longrightarrow \tilde{L}(\maA,\gp_h,h)$
by the formula
\begin{multline*}
T_{g_0}(b_0h_0,\, b_1h_1,\, \ldots,\, b_nh_n) = (hg_0h_0^{-1}(b_0)g_0,
\, g_0(b_1)g_0h_1,\, \ldots\, ,\\ g_0h_1\ldots h_{i-1} (b_{i-1})g_0h_1
\ldots h_i,\, \ldots,\, g_0h_1 \ldots h_{n-1}(b_n) g_0h_1\ldots h_n).
\end{multline*}
Then varify directly that
\[T_{g_0}b_h-bT_{g_0}=0.\]

Let us define the map $T=\frac{1}{|\gp_h|}\sum_{g_0 \in \gp_h}
T_{g_0}$. Clearly
$T$ maps $L(\maA.\gp_h,h)$ to $\tilde{L}(\maA,\gp_h,h)^{\gp_h}$. We
next observe that $\alpha\circ T =Id_{\maH_*(\maA)_{\gamma}} \text{
and } T \circ \alpha = Id_{\tilde{L}(\maA,\gp,h)^{\gp_h}}$ and that $T=\alpha^{-1}$ and therefore must be a chain map. Thus we have established the following commutative diagram of qusi-isomorphisms.
\begin{equation*}
	\xymatrix{ {\maC_*(\maA)^{\gp_h}_h} 
	\ar[dr]^{\pi^{\gp_h}} &{\tilde{L}}(\maA,\gp,h)^{\gp_h}
	\ar[l]^{\pi_{\gp}} \ar@<1ex>[r]^{\alpha} \ar@{=}[d] & \maH_*(\maA)_{\gamma}
	\ar@<1ex>[l]^T \\ & {\tilde{L}}(\maA,\gp_h,h)^{\gp_h}
	\ar@<1ex>[r]^{\alpha _{\gp_h}} & L(\maA,\gp_h,h).
	\ar@<1ex>[l]^T \ar@{^{(}->}[u]}
\end{equation*}

Here $F:=\alpha \circ \pi^{\gp}:\maC_*(\maA)_h^{\gp_h}
\longrightarrow \maH_*(\maA)_{\gamma}$ is of the form
\begin{eqnarray}
F(a_0,a_1,\ldots,a_n)=(a_0h,a_1e,\ldots a_ne).
\end{eqnarray}

At the same time the inverse qusi-isomorphism 
$G:=\pi_{\gp}\circ T:\maH_*(\maA)_{\ip{\gamma}} \longrightarrow
\maC_*(\maA)_h$ is given by the formula \eqref{quasi-isomorphism}
\end{proof}

Thus as a consequence we  immediately have that
\[\maH_*(\maA)_{\ip{\gamma}} \simeq \maC_*(\maA)_h^{\gp_h}\simeq Tor_*^{A^e}(\maA, \maA_h)^{\gp_h}.\]

As already mentioned earlier the computation  shall be of interest in case of $\maA=\smooth(M)$. We now begin our efforts to show that in this case homology of  $C_*(\smooth(M))_h$ can be described by  differential forms on the fixed point sets  $N^h$.

\section{Local computations}\label{section-LAR}
Let us now specialize to the case when the algebra $\maA=\smooth(V)$ is the algebra of smooth functions on a vector space $V$. and $\gamma:V\rightarrow V$ is a linear transformation. In this  section we shall identify our twisted Hoschild homology that is the homology of the complex $\maC_*(\maA)_{\gamma}$ with the differential forms on the fixed point  subspace $V^{\gamma}$. To this end it is  it is most convienient to  introduce the language of Kasul complexs.

\subsection{Koszul Complex}\label{sub-KC}
Let $R$ be a commutative ring. Let $f_1,f_2 \ldots f_q \in R$. Let
$\{v_j\}$ be a basis for $\CC^q$. We define the Kasul complex of $R$
generated by $f_1,f_2,\ldots,f_q$ by
\begin{multline*}
  \maK_l(R:f_1,f_2,\ldots,f_q)=R\otimes \wedge^l\CC^q \\
  \delta(r\otimes v_{i_1}\wedge v_{i_2}\wedge \ldots \wedge
  v_{i_l})=\sum_{j=1}^{j=q}(-1)^j(rf_j \otimes v_{i_1} \wedge v_{i_2} \ldots
  \wedge\widehat{ v_{i_j}} \wedge \ldots \wedge v_{i_l}).
\end{multline*}
This  differential arises naturally from a simplicial  module structure. We observ3e the following properties:
\begin{lem}\label{lem_kas}
\begin{enumerate}
\item\label{Kaszul1}
Let $R$, $R'$ be two algebras over $\CC$. with $S \subset R$ and $S'
\subset R'$ be subsets. Denote by $S \coprod S'=S \otimes 1 \cup
1\otimes S' \subset R\otimes R'$. Then
\[
\maK_*(R\otimes R': S \coprod  S')= \maK_*(R:S)\otimes  \maK_*(R':S').
\]

\item\label{Kaszul2}
 Let $V=\RR^n$. Let $\maA= C^{\infty}(V)$, and let $X_i$ be the coordinate functions on $V$. Then
$$
\maH_q(\maK_*(\maA:\{X_i\}_{i=1}^n)=
\begin{cases}0& q>0\\
  \CC & q=0.
\end{cases}
$$
\end{enumerate}
\end{lem}
\begin{proof}
 Since the  differential  comes from a simplicial object structure  the first fact is a consequence of the Eilenberg-Zibler theorem.

 The second fact follows from Poincare lemma and properties of Fourier transform.
\end{proof}

\subsection{Linear Action on $\RR^n$}\label{sub-LAR}
Given a linear transformation  $\gamma $ of a real vector space $V$ we decompose $V$ into fixed point subspace $V^{\gamma}$ and an invariant complement $(1-\gamma)V$,
\[V=V^{\gamma}\oplus (1-\gamma)V.\]

We  come now to the main result of this section which is the local version of our desired result.
\begin{theo}
Let $\gamma$ be a linear automorphism of the algebra
$\maA=\smooth(V)$. The homology of the twisted complex $
(C_*(A)_{\gamma}, b_{\gamma})$ is then given by the space of forms
on the fixed point $V^{\gamma}$
\[H_q(C_*(A)_{\gamma})\simeq\Omega^q(V^{\gamma}),\]
 and the identification is $\gp$ equivariant.
\end{theo}
\begin{proof}
Let ${e_i}$ be a basis of $V$ such that  $e_i\in V^{\gamma}$ for  $1\leq i\leq m$ and  $e_i\in (1-\gamma)V$ for $m+1\leq i\leq n$.  Let $X_i$ denote the corresponding coordinate functions.
 
 Consider the
Koszul complex,
\[\maK_*(\maA^e:\{X_i\otimes1-1 \otimes  X_i\}_{i=1}^n)\]
which is a projective resolution of $\maA$ over $\maA^e$. And hence
by Lemma \ref{Tor},
\[\maC_*(\maA)_{\gamma} \simeq _{qi} \maK_*( \maA^e:\{X_i\otimes1-1 \otimes
X_i\}_{i=1}^n)\otimes_{\maA^e} \maA_{\gamma}.\]

Observe that $\maK_*(\maA^e:\{X_i \otimes 1-1 \otimes
X_I\}_{i=1}^n) \otimes_{\maA^e} \maA_{\gamma}$ can be identified
with another Koszul complex namely, $\maK_*(\maA: \{
X_i-\gamma(X_i)\})$ via the map
\[\psi(a\otimes b \otimes
v_{i_1}\wedge\ldots \wedge v_{i_l}\otimes c)= a\gamma(b)c\otimes
v_{i_1}\wedge \ldots \wedge v_{i_l}.\] 

This results in the following
diagram.
\begin{center}
  $ \xymatrix { {\ldots} \ar[r] & {(\maA^e\otimes \bigwedge ^lV)\otimes
      _{\maA^e} \maA_{\gamma}} \ar[r]^{\delta \otimes 1 } \ar[d]
    &{(\maA^e\otimes \bigwedge ^{l-1}V)\otimes _{\maA^e} \maA_{\gamma}}
    \ar[d]\ar[r] &{\ldots} \\
    {\ldots} \ar[r] &{(\maA\otimes \bigwedge ^lV)}
    \ar[r]^{\delta} & {(\maA\otimes \bigwedge ^{l-1})} \ar[r]&{\ldots}\\
  } $
\end{center}

Let $R=C^{\infty}(V^{\gamma})$ be smooth functions on the fixed
point manifold $V^{\gamma}$ and let $R'= C^{\infty}((1-\gamma)V)$ be smooth
functions on
the invariant compliment of $V^{\gamma}$.

Since $\maA=C^{\infty}(V)= C^{\infty}(V^{\gamma}) \otimes
C^{\infty}((1-\gamma)V)=R\otimes R'$, by Lemma \ref{lem_kas}.\ref{Kaszul1},
\begin{multline*}
\maK(\maA:\{X_i-\gamma(X_i)\}_{i=1}^{n+q})=\maK(R:\{X_i-\gamma(X_i)\}_{i=1}^m) \otimes \\
 \maK(R': \{X_i-\gamma(X_i)\}_{i=m+1}^{n+q})\\
\simeq \maK(R:\{0\}) \otimes
 \maK(R': \{X_i\}_{i=m+1}^{m+q}).
\end{multline*}
Now by applying Lemma \ref{lem_kas}.\ref{Kaszul2} to $\maK(R':
\{X_i\}_{i=m+1}^{m+q})$, we have
\begin{eqnarray}\label{linear twisted complex} \maK(\maA:
  \{X_i-\gamma(X_i)\}_{i=1}^{n+q})=\maK(R:\{0\})
  \simeq\Omega(V^{\gamma}).
\end{eqnarray}

Here the last identification is due to the following theorem, 
which we formulate only in the $\sigma$--compact case, for 
simplicity.

\begin{theo} [Connes' HKR Theorem] \label{HKR}
Let $X$ be a smooth, $\sigma$--compact manifold.  Then the Hochschild
homology of the algebra $\maA =\smooth(X)$ is given by the
differential forms on $X$. The map
\[\chi_k(a_0\otimes a_1 \otimes  \ldots \otimes  a_k) 
\rightarrow a_0da_1da_2 \ldots da_k\]
induce4s an isomorphism
\[HH_k(\smooth(X)=\Omega^k(X)\]
\end{theo}
\begin{proof}
Let us consider $X=\mbR^n$, which is sufficient to prove the
equality in \ref{linear twisted complex}. It is easily seen that the
map $\chi$ is a chain map. We define the inverse map (which is only
well define on the homology)
\[E_*:\Omega^*(\mbR^n)\rightarrow \maH_*(\smooth(\mbR^n))\]
\[E_k(a_0da_1da_2\ldots da_k) = \frac{1}{k!}\sum_{\pi \in S_k} 
sign(\pi)(a_0\otimes a_{\pi(1)}\otimes a_{\pi(2)} \otimes \ldots
\otimes a_{\pi(k)}).\] Then the following diagram commutes.
\begin{equation*}
\xymatrix{ {\Omega^k(\mbR^n)}\ar@<0.5ex>[r]^{\chi_k} \ar[d]^0 &
{\maH_k(\maA)} \ar[d]^b \ar@<0.5ex>[l]^{E_k}\\
{\Omega^{k-1}(\mbR^n)}\ar@<0.5ex>[r]^{\chi_{k-1}} &
{\maH_{k-1}(\maA).}
\ar@<0.5ex>[l]^{E_{k-1}} }
\end{equation*}
Then the map $\chi_k$ induces an isomorphism
\[HH_k(\smooth(\mbR^n))\simeq \Omega^k(\mbR^n).\]

The case of a closed manifold $X$ will be proved in the next
section.
\end{proof}
Thus combining all this information we get the following diagram.
\begin{equation*}\xymatrix {
{\maC_*(\maA)_{\gamma}} \ar[r]\ar[d]^{\sigma}
&{\maH_*(\maA)\otimes_{\maA^e}\maA_{\gamma}} \ar[r]
&{Tor^{\maA^e}(\maA,\maA_{\gamma})} \ar[ld]\\
{\Omega^*(V^{\gamma})} & {\maK(\maA^e:\{X_i\otimes 1-1\otimes X_i\})} \ar[l]\\
} 
\end{equation*}
We note that all the maps involved are $\gp$ equivariant.
\end{proof}
The following result now follows immediately by Proposition
\ref{cp}.
\begin{cor}\label{fixed_point}
Let $\gp$ be a finite group acting linearly on $\maA=\smooth(V)$ and
  $\maB=\maA\rtimes \gp$. And let $\gamma\in \gp.$ Then the $\langle \gamma
  \rangle$ component of Hochschild homology $HH(\maB)_{\gamma}$ is given by
\[HH_*(\maB)_{\gamma}=\Omega(V^{\gamma})^{\gp_{\gamma}}.\]
\end{cor}


\section{Localization with Groupoids and Sheaf}\label{section-LGS}
We return to the situation where $\maA=\smooth(M)$ si athe algebra of smooth functions over a closed manifold $M$ and $\gp$ acts by  diffeomorphisms.
In this section we shall obtain the ``normalized complex''  for the Hoschild homology of our cross-product algebra $\smooth(M)\rtimes  \gp$. To this end we identify  the cross-product algebra as usual to the convolution algebra over the  transformation groupoid $M
\rtimes \gp$. The normalized complex thus provides a complex of sheaves  that reduces the calculation of Hoschild homology  to the previous example of linear action on euclidian space.

Recall that the transformation groupoid as a space is just
$G=M\times \gp:=\{(x,g)|x\in M \, ,\,g\in\gp\}$ with units $M$ and the source, the
range and the composition maps given by
\[ s(x,g)=g^{\!-\!1}.x ~~~~~\text{and}~~~~~  r(x,g)=x\]
\[(x,g).(y,h)=(x,gh) \hskip.3in \text{when }\,\, x=g.y\]
and the inverse defined by
$(x,g)^{\!-\!1}:=(g^{\!-\!1}x,g^{\!-\!1})$. The convolution algebra
of $G$ is the space $\smooth(G)$ with the product given by
\[f_1*f_2(x,g):=\sum_{h\in \gp} f_1(x,h).f_2(( x,h)^{\!-\!1}.(x,g))\]

We can identify the cross-product algebra $\maB=\smooth(M)\rtimes
\gp$ with the convolution algebra $\smooth(G)$, by the map
$$\smooth(G)\ni f\!\mapsto\!\sum_{g\in \gp}f_gg\in \smooth(M)\rtimes
\gp$$ with $f_g(x)=f(x,g)$.

This leads to another description of the Hochschild and cyclic
complexes for $\maB$. First define $\Phi :\maB^{\otimes n}\simeq
\!\smooth(G)^{\otimes n}\!\rightarrow\! \smooth(G^n)$ as a map of nuclear 
Fr\'ech\`et algebras by
\[\Phi(f_0\otimes f_1\otimes \ldots
\otimes f_{n-1})(a_0,a_1,\ldots a_{n-1})=f_0(a_0).f_1(a_1)\ldots
f_{n-1}(a_{n-1}),\]
where $f_i \in \smooth(G)$ and $a_i:=(x_i,g_i)\in
G$. Then $\Phi$ is an isomorphism \cite{Grothendieck}. With
this in mind, the Hochschild differential takes the form
\begin{multline*}
  b(F)(a_0,a_1,\ldots ,a_{n-1}) \\
  = \sum_{i=0}^{n-2}(-1)^i\sum_{\gamma\in \gp} F(a_0,a_1,\ldots
  ,a_{i-1},(x_i,\gamma),
  (\gamma^{\!-\!1}x_i,\gamma^{\!-\!1}g_i),a_{i+1},\ldots ,a_{n-1})\\
  +(-1)^{n-1}\sum_{\gamma\in \gp}F((\gamma^{-1}x_o,\gamma^{-1}g_0),a_1,\ldots,a_{n-1},
  (x_o,\gamma)).
\end{multline*}
It turns out (see \cite{Br-Ni}) that the normalized complex for the convolution algebra $\smooth(G)$   provides a complex of sheafs of germs of smooth functions over certain ``loop spaces''. To observe this let us denote by  $B^{(n)}\subset
G^{n+1}:=\{(a_0,a_1,\ldots,a_n)\,|\,s(a_i)=r(a_{i+1}) \forall 0\leq
i\leq n\}$  the space of loops in $G^{n+1}$ Here and through this section we shall use the convention $n+1=0$ while talking about the  loops space $B^{(n)}$.
For example, $B^{(0)}=\{(x,g)\in G|\,g^{-1}x=x\}$ is the disjoint
union of the fixed point manifolds.

 Lets
denote by $\maI_n$ the submodule (in fact an ideal) of
$\smooth(G^{n+1})$ defined by
\[\maI_n=\{f\in\, \smooth(G^{n+1})~|\hskip 0.2in \text{supp}(f)\cap
B^{(n)}=\emptyset\}.\] Then we have the following result.
\begin{prop}\label{sheafy}
The complex $(\maI_n,b)$ is an acyclic subcomplex of $\maH_n(\maB)$
and hence the quotient map $\maB^{\otimes n+1}\rightarrow
\maB^{\otimes n+1}/\maI_n$ is a quasi-isomorphism.
\end{prop}
\begin{proof}
Let $\maU_j^{n+1}\subset G^{n+1}=\{(a_0,a_1,\ldots a_n) |s(a_i)\neq
r(a_{i+1})\}$ be a sequence of open sets. This yields an filtration
of the complex $\maB^{\otimes n+1}$ as
\[F_i=F_i\maB^{\otimes n+1}=\sum_{j=0}^i\{f~|~\text{supp}(f)\subseteq
  \maU_j^{n+1}\}.\]
Also, let $F_i\maB^{\otimes n+1}:=F_n\maB^{\otimes n+1}$ for
  $i>n$. Then $F_i$ is a filtration on the complex $(\maI_n,b)$. We can check that in the corresponding spectral sequence $E^0_n=(\oplus
  F_{i+1}/F_i,b)$ is acyclic.
\end{proof}
Thus the Hochschild chains are reduced from functions on $G^{n+1}$
to germs of function near $B^{(n)}$. Denote by
$i_n:B^{(0)}\rightarrow G^{n+1}$ the loop inclusion map
$(x,g)\rightarrow ((g^{\!-\!1}x,e),(g^{\!-\!1}x,e),\ldots
(g^{\!-\!1}x,e))$. Let
\[\maF_n=i_n^{\!-\!1}\mathcal{C}^{\infty}(G^{n+1})\]
be the pull-back
of the sheaf of smooth functions on $G^{n+1}$ to $B(n)$.. 

To make a more geometrical identification  let us benote by  the projection
$\pi:B^{(n)}\rightarrow B^{(0)}$ defined by
\begin{align*}
\pi((x_0,g_0),(x_1,g_1),\ldots ,(x_n,g_n))&=(x_0,g_0).(x_1,g_1)
\cdots (x_n,g_n)\\
&=(x_0,g_0.g_1\ldots g_n)
\end{align*}
is a local homeomorphism.
It is in fact a diffeomorphism on when restricted to each connected
component. For a fixed $u=(g_0,g_1,\ldots, g_n)\in \gp^{n+1}$, let
$\gamma=g_0g_1\ldots g_n$ and let $B^{(n)}_u$ denote the component
of $B^{(n)}$ consisting of $\{(x_0,g_0),(x_1,g_1),\ldots,
(x_n,g_n)\in B^{(n)} \}$. Then $\pi$ maps $B^{(n)}_u$
diffeomorphically onto $B^{(0)}_{\gamma}\backsimeq M^{\gamma}$. Noting again that $B^{(0)}$ is
just the disjoint union of the fixed point manifolds, we have the
following.
\begin{prop}\label{global_sect}
The pull back of the sheaf $\maF_n$ under the local diffeomorphism
$\pi:B^{(n)}\rightarrow B^{(0)}$ gives an isomorphism
\[\gp(B^{(n)},\pi^{\!-\!1}\maF_n)=\maB^{\otimes n+1}/\maI_n,\]
as vector spaces.
\end{prop}
\begin{proof}
Every point $u\in B^{(n)}$ has a neighborhood $W_u$ in $G^{n+1}$
which is diffeomorphic to $V^{n+1}$ for some Small enough
neighborhood $V$ of $s(a_n)=g_n^{\!-\!1}a_n$ in $G$. One possible
identification can be $W_u=(V,g_0)\times (g_0V,g_1)\times\ldots
\times (g_0g_1\ldots g_{n-1}V,g_n)$. 
Then this diffeomorphism gives a map $\psi:\smooth(V^{n+1})\rightarrow
\smooth(W_u)$. Thus after covering $B^{(n)}$ with small enough
neighborhoods like $W_u$, would match the global sections in
$\gp(B^{(n)},\pi^{\!-\!1}\maF_n)$ to functions smoothly supported in
some neighborhood of $B^{(n)}$ in $G^{n+1}$.
\end{proof}

Using the proposition above, one may obtain a sheafified version of
Proposition \ref{cp} with an induced differential on $\maF_n$.
\begin{remark}
At each $a=(x,\gamma)\in B^{(0)}$, the map
$\phi:\smooth(V^{n+1})\rightarrow\smooth(W)$ defined above induces an
isomorphism $\phi_a$ on the stalks $(\maF_n)_a\simeq
\smooth(M)_{\gamma^{\!-\!1}x}^{\otimes n+1}$ to the stalk
$(\smooth(G^{n+1}))_{(a,a\ldots a)}$ of the sheaf of smooth
functions which are germs of smooth functions on $G^{n+1}$ at
$(a,a,a\ldots,a)$. And this isomorphism on the stalk is given by
the formula
\[\phi_a(f_0\otimes f_1\otimes \ldots \otimes f_n)
\rightarrow(f_0\gamma\otimes f_1e\otimes \ldots \otimes f_ne). \]
\end{remark}
Under $\phi$ at every point $(x,\gamma)\in B^{(0)}$ the stack
$(\maF_n)_{(x,\gamma)}$ have a simplicial structure given by the
differential $b_{\gamma}$ from Equation (\ref{twisted}) and the map
$\phi$ is a chain map. Using Corollary \ref{fixed_point}, we can
compute the homology of each stalk. In the following section, we study
the properties of sheaves of complexes for which a quasi-isomorphism
on stalks give global quasi-isomorphisms on complexes of global
sections.

\subsection{Sheaves}\label{sub-S}

Let $X$ be a compact Hausdorff space. A sheaf $\mathcal{S}$ on $X$
is called \textit{flabby} if for any open set $U \subseteq X$, the
restriction map
$$\textrm{res}_{XU}:\Gamma(X,\mathcal{S})~\longrightarrow~\Gamma(U,\mathcal{S})$$ is
surjective.

Let $\maK$ be a sheaf of unital algebras on $X$. In particular, we
assume that the restriction maps are unital ($1 ~\in~
\Gamma(X,\maK)$ and the image under restriction
$\textrm{res}_{XU}(1)$ is a unit in $\Gamma(U,\maK)$). We say that
the sheaf $\maK$ has a partition of unity, if for any locally finite
cover $U_{\alpha}$ of $X$ there exists
$\Phi_{\alpha}~\in~\Gamma(X,\maK)$ such that
\begin{enumerate}
\item $\textrm{supp}(\Phi_{\alpha})~\subseteq~U_{\alpha}$.
\item $\sum_{\alpha} \Phi_{\alpha}~=~1$.
\end{enumerate}

Let $\mathcal{F}_i$ be flabby sheaves of $\maK$ modules. By this we
mean that for any open set $U~\subseteq~X$, the space of sections
$\Gamma(U,\mathcal{F}_i)$ is a modules over$\Gamma(U,\maK)$. Let
$$\xymatrix{
\mathcal{F}~=~\mathcal{F}_0&\mathcal{F}_1\ar[l]_{d_0}&\mathcal{F}_2\ar[l]_{d_1}&\dots\ar[l]_{d_2}}$$
be a sheaf of complexes such that each $d_i$ is $\maK-$linear.

\begin{lem}\label{flabby}: Let  a be flabby sheaf of complexes $\mathcal{F}$ over
a sheaf of unital algebra $\maK$. If $\maK$ has partition of unity
then
$$\Gamma(X,\mathcal{H}_{*}(\mathcal{F})) ~\simeq~ \mathcal{H}_{*}(\Gamma(X,\mathcal{F})).$$
\end{lem}
\begin{proof} Let $b \in \Gamma(X,\mathcal{F}_{*})$ be such that
$d_{*}^{X}(b)~=~0$. So $[b] \in
\mathcal{H}_{*}(\Gamma(X,\mathcal{F}))$. For any open set $U
\subseteq X$, define
$$b^U=\textrm{res}_{XU}(b)~ \text{ and }~[b^U] \in
\Gamma(U,\mathcal{H}_{*}(\mathcal{F})).$$ Then for any open cover
$\left\{U_{\alpha}\right\}$ of $X$,
$$\textrm{res}_{U^{\alpha},U^{\alpha} \bigcap
U^{\beta}}\left([b^{U_{\alpha}}]\right)~=~[b^{U_{\alpha} \bigcap
U_{\beta}}]~=~\textrm{res}_{U_{\beta},U_{\alpha} \bigcap
U_{\beta}}\left([b^{U_{\beta}}]\right).$$ Hence there is a global
section $\Phi(b) \in
 \Gamma(X,\mathcal{H}_{*}(\mathcal{F})$. It follows from the
 definition that $\Phi(b)$ depends only on the class of $[b] \in
 \mathcal{H}_{*}(\Gamma(X,\mathcal{F}))$.

The map $\Phi$ above
 $\mathcal{H}_{*}(\Gamma(X,\mathcal{F}))~\rightarrow~\Gamma(X,\mathcal{H}_{*}(\mathcal{F}))$
 is always defined. We now prove that under our assumptions on $k \in
 \mathcal{F}$, it is in fact an isomorphism.

We construct an inverse as follows. Choose an open cover
 $U_{\alpha}$ and a subordinate partition of unity
 $\Phi_{\alpha}$. Let $a \in
 \Gamma(X,\mathcal{H}_{*}(\mathcal{F}))$ and $a_{\alpha} = \textrm{res}_{XU_{\alpha}}(a) \in
 \Gamma(U_{\alpha},\mathcal{H}_{*}(\mathcal{F}))$. Pick
 representatives $b_{\alpha} \in
 \Gamma(U_{\alpha},\mathcal{F}_{*})$ with $[b_{\alpha}] =
 a_{\alpha}$. Since $\mathcal{F}_{*}$ is flabby, there exists a $\widetilde{b_{\alpha}} \in
 \Gamma(X,\mathcal{F}_{*})$ so that $\textrm{res}_{XU_{\alpha}}(\widetilde{b_{\alpha}}) =
 b_{\alpha}$.

Let $b_{1} = \sum \Phi_{\alpha}\widetilde{b_{\alpha}} ~\in~
 \Gamma(X,\mathcal{F}_{*})$. Then $d_{*}(b_{1})|_{U_{\alpha}} =
 0$ and hence $d_{*}(b) = 0$. For any other choice of
 representatives and partition of unity or cover, we can see that
 the class of $b_{1}$, $[b_{1}] \in
 \mathcal{H}_{*}(X,\mathcal{F})$, is independently defined. We say
 that $\sigma(a) = b_{1}$, and from definition of
 $\sigma$ and $\Phi$, they are inverses of each other.
 \end{proof}

\section{Smooth Manifolds}\label{section-SM}

Let $M$ be a closed manifold. We are now ready to find the
Hochschild and cyclic homology of $\smooth(M)\rtimes  \gp$ using
localization argument and the results for $\RR^n$ of the previous
sections.

Recall that for every conjugacy class $\langle \gamma \rangle $ of
the group $\gp$, there is a subcomplex of $\maH_*(\maB)$ given by
$\maH_*(\maB)_\gamma= \{(a_{g_0}g_0\otimes a_{g_1}g_1\otimes \ldots
\otimes a_{g_n}g_n)|g_1.g_2\ldots g_n.g_0 \in\langle \gamma
\rangle\}$ and which yields the decomposition :
$\maH_*(\maB)=\bigoplus_{\langle \gamma
\rangle}\maH_*(\maB)_\gamma.$

Fix a conjugacy class by choosing a representative $\gamma$. We
denote by $\gp_{\gamma}$ the stabilizer of $\gamma$ in $\gp$. Also
if $\gamma'\in\langle \gamma \rangle$, then we denote by $ S_{\gamma
\gamma'}=\{k\in\gp\, |\,\gamma =k\gamma' k^{-\!1}\}$. Define a map
$\chi_*^{\gamma}:\maH_*(\maB)_{\gamma}\rightarrow
\Omega(M^{\gamma})$ from the Hochschild complex to the forms on the
fixed point manifold, by
\begin{multline*}
\chi_m^{\gamma}(a_{g_0}g_0\otimes a_{g_1}g_1\otimes \ldots \otimes
a_{g_n}g_n):=\\
\left(\frac{1}{|\gp_{\gamma}|}\sum_{h\in S_{\gamma\gamma'}}\gamma
hg_0^{-1}(a_0)dh(a_1)dhg_1(a_2) \ldots dhg_1g_2\ldots
g_{n-1}(a_n)\right)_{|_{M^{\gamma}}},
\end{multline*}
where $\gamma'=g_1g_2\ldots g_ng_0$. Then $\chi_n^{\gamma}\circ b=0$
and $\chi_n^{\gamma}\circ B=d_{M^{\gamma}}$ is the de Rham
differential on $M^{\gamma}$.
\begin{theo}\label{HH_crossproduct}
The Hochschild homology of the complex $\maH_*(\maB)_{\gamma}$ is
isomorphic to the $\gp_{\gamma}$ invariant forms on the fixed point
manifold $M^{\gamma}$. More precisely, the map
\[\chi_*^{\gamma}: \maH_*(\maB)_{\gamma} \rightarrow \Omega^*(M^{\gamma})^{\gp_{\gamma}}.\]
is a quasi-isomorphism between $(\maH_n(\maB)_{\gamma},b)$ and
$(\Omega^n(M^{\gamma})^{\gp_{\gamma}},0)$
\end{theo}
\begin{proof}
By Proposition \ref{sheafy}, the map
  $\Phi:\maH_*(\maB)_{\gamma}\rightarrow
  (\gp(B^{(0)}_{\gamma},\maF_n^{\gp_{\gamma}}),b_{\gamma})$ is an isomorphism. The sheaves $\maF_n$ are
flabby and are modules over the sheaf of invariant functions
$\mathfrak{K}=\mathfrak{C}^{\infty}(M)^{\gp}$. This is the sheaf
associated to the presheaf $\maU\rightarrow \smooth(\maU)^{\gamma}$.
Also the differentials $b_{\gamma}$ are
$\mathfrak{C}^{\infty}(M)^{\gp}-$linear. Similarly, the sheaf of
differential $n$-forms $\Omega_n$ too is a flabby sheaf of
$\mathfrak{K}$ modules. Then by Theorem \ref{HKR}, the chain map
$\tilde{\chi}_n(a_0,a_1,\ldots a_n)\rightarrow a_0da_1\ldots da_n$
is in fact an quasi-isomorphism on each stalk and inverse (defined
only on homology) given by
\[E_k(a_0da_1da_2\ldots da_k)=\frac{1}{k!}\sum_{\pi \in S_k}sign(\pi)
(\tilde{a}_0\gamma \otimes \tilde{a}_{\pi(1)}e\otimes
\tilde{a}_{\pi(2)}e \otimes \ldots \otimes \tilde{a}_{\pi(k)}e).\]
By Lemma \ref{flabby}, $\tilde{\chi}$ is an isomorphism from
\[\gp(B^{(0)}_{\gamma},\maF_n)\rightarrow \Omega(M^{\gamma}).\]
Thus the theorem follows by observing that
$\chi_n^{\gamma}=\tilde{\chi}\circ\Phi$.
\end{proof}
\begin{theo}\label{HC_crossproduct}
\[HC_k(\maB)_{\gamma}=\Omega^k(M^{\gamma})^{\gp_{\gamma}}\bigoplus_{j>0}
H_{\textrm{de~Rham}}^{k-2j}(M^{\gamma})^{\gp_{\gamma}}.\]
\end{theo}
\begin{proof}
Since $\chi:(B_*(\maB)_{\gamma},b,B)\rightarrow
(\Omega^*(M^{\gamma})^{\gp_{\gamma}},0,d)$ is a map of mixed
complexes which is an isomorphism on the columns by Theorem
\ref{HH_crossproduct} above, $\chi^{\gamma}$ must be an isomorphism
on the total complexes of these mixed complexes.
\end{proof}
\begin{cor}
  \[HP_k(\maB)_{\gamma}=\sum_{j\in\ZZ}
  H_{\textrm{de~Rham}}^{k-2j}(M^{\gamma})^{\gp_{\gamma}}.\]
\end{cor}
\begin{proof}
Since $HH_n(\maB)_{\gamma}=0$ for $n>\textrm{dim}(M)$, the
periodicity map $S:HC_{n+2}(\maB)_{\gamma}\rightarrow
HC_n(\maB)_{\gamma}$ is an isomorphism by the SBI-exact sequence. In
particular,
$\underleftarrow{\textrm{lim}}HC_*(\maB)_{\gamma}=\sum_{j\in\ZZ}
  H_{\textrm{de~Rham}}^{k-2j}(M^{\gamma})^{\gp_{\gamma}}$. Also
$\underleftarrow{\textrm{lim}}^1CC_*(\maB)_{\gamma}=0$.
\end{proof}

\section{Topologically Filtered Algebra}\label{section-TFA}

The algebra of complete symbols with the topology that is described
below, is not a topological algebra in the above sense because the
multiplication is not jointly continuous. Hence we need a larger
category, namely that of topologically filtered algebra defined in
\cite{BenameurNistor}. Let $\maA$ be a algebra  with filtration
$\maA=\bigcup_{p\in\ZZ}\maF_p\maA$. That is to say, each
$\maF_p\maA$ is a subspace $\maF_p\maA\subset \maF_{p+1}\maA$ and
the multiplication map takes  $\maF_p\maA\times
\maF_q\maA\rightarrow \maF_{p+q}\maA$. We would say that $\maA$ is a
filtered algebra for short.

Since, by definition, $\maF_0\maA$ is a subalgebra of $\maA$ and
$\maF_{-1}\maA$ is an ideal of $\maF_0\maA$,
$\maM_0:=\maF_0\maA/\maF_{-1}\maA$ is naturally an algebra and each
$\maM_p^j:=\maF_p\maA/\maF_{p-j}\maA$ is a module over $\maM_0$.
Similarly, $\maI:=\bigcap_p\maF_p\maA$ is an ideal in $\maA$.  We
call our algebra a symbol algebra if $\maI=0$. We only consider
symbol algebra for now. That is, if $\maA'$ is any filtered algebra,
we consider the algebra $\maA=\maA'/\maI$.
\begin{definition}
We say that a filter algebra is \emph{topologically filtered} if
\begin{enumerate}
\item $\maA$ is a symbol algebra, that is $\maI:=\bigcap_p\maF_p\maA=0$.
\item  Each $\maM_p^j=\maF_p\maA/\maF_{p-j}\maA$ is a nuclear 
Frechet space for all $p$ and each $j$.
\item Each module map $\maM_p^j\hat{\otimes} \maM_q^j\rightarrow
\maM_{p+q}^j$ induced by the multiplication in $\maA$ is continuous.

We call an element $P\in \maF_1\maA$ {\em elliptic} if it is invertible and
$P^n\in\maF_n\maA$ for all integers $n$.
\item \label{module} There exists an elliptic element such that the map
$\maF_n\maA/\maF_{n-1}\maA\ni[P^n]\rightarrow P^n\in \maF_n\maA$ gives
a linear splitting of
$\maF_{n-1}\maA\hookrightarrow\maF_n\maA\rightarrow
\maF_n\maA/\maF_{n-1}\maA$.
\end{enumerate}
\end{definition}

If $\maA$ is a topologically filtered algebra then using the existence
of an elliptic element, we have
\begin{align*}
 \maF_p\maA&=\maF_{p-1}\oplus\maF_p\maA/\maF_{p-1}\maA\\
&=\maF_{p-2}\oplus\maF_p\maA/\maF_{p-1}\maA\oplus\maF_{p-1}\maA/\maF_{p-2}\maA\\
&=\maF_{p-2}\oplus\maF_p\maA/\maF_{p-2}\maA.
 \end{align*}
Since $\maI=0$, by repeating the iterations, we obtain
$$\maF_p\maA=\underleftarrow{\textrm{lim}}\,\maF_p\maA/\maF_{p-k}\maA=\prod_{j\leq
p} \maF_j\maA/\maF_{j-1}\maA.$$ We use this description to endow
$\maF_p\maA$ with the projective limit topology from the Frechet
topologies on $\maF_p\maA/\maF_{p-k}\maA$.  $\maA$ is then endowed
with the inductive limit topology from
$\maA=\underrightarrow{\textrm{lim}}\maF_p\maA=\prod_p
\maF_p\maA/\maF_{p-1}\maA$.  Since strong inductive limits of nuclear spaces is
nuclear, the
topology on $\maA$ is nuclear. We call the resultant topology the weak
topology on $\maA$. (The inductive limit topologies are `strong'
topologies. The nomenclature here is to distinguish the inductive
limit topology on $\maA$ with yet another topology which is stronger.)
Since $\maA$ is a filtered algebra $\maA_0:=\maF_0\maA$ is a
subalgebra and further $\maA_{-1}:=\maF_{-1}\maA$ is an ideal in
$\maA_0$. By Axiom
\ref{module}, each
 $A_i=\maF_i\maA/\maF_{i-1}\maA$ is a module over $A_0=\maA_0/\maA_{-1}$
 generated by a single element $[P^i]$.
\begin{prop}\label{product}
The multiplication in $\maA$ is separately continuous with respect
to the weak topology. The multiplication is jointly continuous on
the subalgebra $\maA_0=\maF_0\maA$.
\end{prop}
\begin{proof}
First we prove that multiplication is jointly continuous on
$\maA_0$. This is an immediate consequence of universal property of
projective limits. Since the multiplication map on
  $\maF_0\maA/\maF_{-j}\maA\hat{\otimes}\maF_0\maA/\maF_{-j}\maA\rightarrow\maF_0\maA/\maF_{-j}\maA$
  is continuous for each $j$, by composition there is a continuous map
  $$\phi_j:\underleftarrow{\textrm{lim}}\,\maF_0\maA/\maF_{-j}\maA\hat{\otimes}
  \underleftarrow{\textrm{lim}}\,\maF_0\maA/\maF_{-j}\maA\rightarrow\maF_0\maA/\maF_{-j}\maA.$$
  Using the universal property of $\underleftarrow{\textrm{lim}}$, there is
  thus a unique continuous map on
  $\phi:\underleftarrow{\textrm{lim}}\,\maF_0\maA/\maF_{-j}\maA\hat{\otimes}
  \underleftarrow{\textrm{lim}}\,\maF_0\maA/\maF_{-j}\maA\rightarrow
  \underleftarrow{\textrm{lim}}\, \maF_0\maA/\maF_{-j}\maA=\maA_0$.
To complete the proof, we must verify that the map $\phi$ is indeed
the multiplication map. But multiplication on $\maA_0$ composed with
the projection $\maA_0\rightarrow \maF_0\maA/\maF_{\!-\!1}\maA$ is
just the map $\phi_j$ above and hence the multiplication must be
$\phi$ by uniqueness (as algebraic maps) on the projective limit.

\center{\small{ $
    \xymatrix{{\maF_0\maA/\maF_{-j}\maA\hat{\otimes}\maF_0\maA/\maF_{-j}\maA}
      \ar[r] &{\maF_0\maA/\maF_{-j}\maA} \\
      {\vdots}\ar[u] &{\vdots} \ar[u]\\
      {\underleftarrow{\textrm{lim}}\,\maF_0\maA/\maF_{-j}\maA
        \hat{\otimes}\underleftarrow{\textrm{lim}}\,
        \maF_0\maA/\maF_{-j}\maA}\ar[u]\ar[r]^{\phi}\ar[uur]^{\phi_j}&
      {\underleftarrow{\textrm{lim}}\, \maF_0\maA/\maF_{-j}\maA}\ar[u]}$
      }}\center

The separate continuity of the multiplication on $\maA$ can be
proved similarly.
\end{proof}
Although a topologically filtered algebra $\maA$ is a nuclear space
and the completed projective tensor product  $\maA^{\hat{\otimes
}n}$ is again nuclear by Proposition \ref{product}, the Hochschild
boundary map $b$ may not be continuous on $\maA^{\hat{\otimes } n}$.
We could define a Hochschild complex for $\maA$ in the following
fashion.

Let
\[\maF'_p:=\sum_{\sum p_i\leq p} \maF_{p_0}\maA\hat{\otimes}
\maF_{p_1}\maA\hat{\otimes} \ldots \hat{\otimes} \maF_{p_n}\maA.\]
Let
$\maF_pC_n(\maA):=\underleftarrow{\textrm{lim}}\maF'_p/\maF'_{p-f}$.
Then the differentials $b$ and $b'$ are continuous on each
filtration. And if $1\in\maF_0\maA$ then the operators $s,t,N$ and
$B$ induce continuous maps on $\maF_pC_n(\maA)$. Denote by
$\maF_pHH_*(\maA)$ the homology of the complex $\maF_pC_*(\maA)$. We
define the Hochschild homology of $\maA$ to be
$HH_*(\maA)=\underrightarrow{\textrm{lim}}\,\maF_pHH_*(\maA)$.

Since it is often useful to use spectral sequences to compute
homologies for such filtered algebras, it is useful to describe the
topology on the associated graded algebra of $\maA$ by using the
identification
\[\graded(\maA)=\underrightarrow{\textrm{lim}}\bigoplus_{p=-N}^N\maF_p\maA/\maF_{p-1}\maA.\]
And the topology on the Hochschild complex for $\graded(\maA)$ is
defined by \[\maH_n(\graded(\maA)):=\underrightarrow{\textrm{lim}}
\left(\bigoplus_{p=-N}^N\maF_p\maA/\maF_{p-1}\maA\right)^{\otimes
n+1}.\] Again the Hochschild boundary map would be continuous with
respect to this topology. The homogeneous components of
$\graded(\maA)$ can be defined by
\begin{align*}\maH_n(\graded(\maA))_p:=&\underrightarrow{\textrm{lim}}\bigoplus_{p=-N}^N
\hat{\otimes}_{k_j}\maF_{k_j}\maA/\maF_{k_j-1}\maA,\\
&\textrm{where}\,-N\leq k_j\leq N \,\textrm{and}\, \sum k_i\leq p.
\end{align*}

The following examples are of interest.  The algebra of complete
symbols $\maA=\pdoM$ on a closed manifold $M$ is a topologically
filtered algebra with the filtration given by the order of an operator
$\maA=\cup_m \Psi^m(m)/\Psi^{-\infty}(M)$.  If $P$ be an elliptic
operator in $\Psi^11(M)$ then the principal symbol map
$\rho:=\sigma(P)\rightarrow P$ defines a splitting of
$\Psi^{m-1}(M)\hookrightarrow\Psi^m(M)\rightarrow
\smooth(S^*(M)\rho_m$. Furthermore each $\Psi^m(M)/\Psi^{m-j}(M)$ is
isomorphic to $\Gamma^{\infty}(S^*M\times \CC^j:S^*M)$ and hence has a
Frechet topology.  As described before we would topologize $\maA$ by
\[\maA=\underleftarrow{\textrm{lim}}\underrightarrow{\textrm{lim}}
\Psi^m(M)/\Psi^{m-j}(M).\]
Its important here to note that the multiplication on the algebra of
complete symbols is not jointly continuous , but only separately
continuous. 

From above it also follows that
\[\graded(\maA)\simeq\oplus_{m\in\ZZ}\smooth(S^*M)\rho^m
=\smooth(S^*M)\otimes \CC[\rho,\rho{\!-\!1}].\]

\section{Homology of Complete Symbols}\label{chapter-hom-symb}
Let $\maA$ be the algebra of complete symbols on a closed manifold
$M$, that is,
$$\maA(M)=\pdoM=\bigcup_m\Psi^m(M)/\Psi^{-\infty}(M)=\bigcup_m\maF_m\maA(M).$$
A group $\gp$ acting smoothly on $M$ induces an action on the
algebra $\maA(M)$ by $(\gamma, A)\rightarrow \gamma A\gamma^{-1}$.
This action preserves the filtration on the algebra $\maA(M)$. Thus
the algebraic cross-product algebra $\maB(M)=\maA(M)\times \gp$ is
again a filtered algebra with induced filtration

\[\maF_m\maB(M)=\maF_m\maA(M)\rtimes \gp\].
We would compute the Hochschild and cyclic homology of $\maB(M)$
using the spectral sequence associated with the above filtration.
But first we prepare some background in the symplectic structure of
the cotangent bundle $T^*M$.

\section{Symplectic Poisson Structure}\label{section-SPS}

\begin{definition}
Let $(X,\omega)$ be a symplectic manifold. A submanifold $i:Y\hookrightarrow X$ is called a \emph{symplectic submanifold} if the
pull-back $i^*\omega$ is a symplectic form on $Y$. Since $i^*\omega$
is certainly a closed form, it is only required that it is
non-degenerate on $Y$ to be a symplectic form.
\end{definition}

A symplectic form on $X$ gives rise to a Poison structure on
$\smooth(X)$. For $f\in\smooth(X,\RR)$, the Hamiltonian vector field
generate by $f$ is the unique vector field $X_f$ such that
$df=i(X_f)\omega$. Poisson structure on $X$ can now be defined by

\[\{f,g\}_X=\omega(X_f,X_g)\quad\forall f,g\in\smooth(X,\RR).\]




\section{Normal Darboux Coordinates}\label{section-NDC}
Let $X$ be a symplectic manifold and $Y$ be a symmetric submanifold
at a point $x\in Y$. The tangent space $T_xX$ decomposes to
$$T_xX=T_xY\oplus T_xY^{\bot},$$
where $T_xX^{\bot}$ is the symplectic orthogonal to $T_xY$. Let
$TY^{\bot}$ be the subbundle on $Y$ with fiber $T_xY^{\bot}$ at each
$x$. We call $TY^{\bot}$ as the \emph{symplectic normal bundle.}

For a submanifold $Y$ of $X$, we define the inhalator $TY^{\perp}$
of $Y$ is the subbundle of the pull-back $TX_{|Y}$ defined by
$$TY^{\bot}=\{\textbf{Y}\in\gp^{\infty}(TX)_{|Y}\,\textrm{such
that}\,\omega_{|Y}(\textbf{X,Y})=0, \, \forall\textbf{X}\in\gp^{\infty}(TY)\}.$$ The following lemma states that
symplectic Poisson structure on $M$ naturally restricts to the
symplectic Poisson structure on a symplectic submanifold $N$.

\begin{lem} A submanifold $Y$ of the symplectic manifold $(X,\omega)$ is
symplectic if and only if $TY^{\perp} \cap TY=0$
\end{lem}
We assume for the rest of the section that $X$ is a symplectic
manifold and $Y$ a symplectic submanifold of $X$.

Let $\mathcal{U}(x_1,\ldots,x_n,\xi_1,\ldots,\xi_n)$ be a Darboux
coordinate chart at $x$ in $X$. We say $\mathcal{U}$ is \emph{normalDarboux chart} if

$$\mathcal{U}\cap Y=\{x_{k+1}=\cdots=x_n=0=\xi_{k+1}=\cdots=\xi_n\}$$ and

$\mathcal{U}\cap Y$ is a trivializing neighborhood of $TY^{\bot}$
such that $(x_{k+1},\ldots,x_n,\xi_{k+1},\ldots,\xi_n)$ give a
trivialization. This is to say

$$\mathcal{U}\cong T(\mathcal{U}\cap Y)^{\bot}.$$

We use the short hand $\mathcal{U}(\bx_1,\bx_2,\bxi_1,\bxi_2)$ with
$\bx_1=(x_1,\ldots,x_k)$ and $\bx_2=(x_{k+1},\ldots,x_n)$ and so on.


We consider $TY^{\bot}$ as the default normal bundle to $Y$ in $X$.
Any local chart at a point $x\in Y$ would trivialize $TY^{\bot}$ on
$Y$. If $\phi:\mathcal{U}(\bx_1,\bx_2,\bxi_1,\bxi_2)\longrightarrow\mathcal{U}({\bf y}_1,{\bf y}_2,\bfeta_1,\bfeta_2)$ be a change of
normal coordinates then it is a bundle map

$$\phi~:~T(\mathcal{U}\cap Y)^{\bot}\longrightarrow
T(\mathcal{U}\cap Y)^{\bot}$$ (and hance is linear in $\bx_2$ and
$\bxi_2$).

\begin{prop}
Let $Y$ be a symplectic submanifold of $X$. Then there exists an
extension $C^{\infty}(Y)\longrightarrow C^{\infty}(X)$ denoted by
$f\rightarrow \hat{f}$ such that

$\{\hat{f},~\hat{g}\}_{X_{|Y}}=\{f,~g\}_Y$ for all $f,~g\in
C^{\infty}(Y)$.
\end{prop}

\begin{proof}
Let $\{U_{\alpha}\}$ be a cover of $Y$ by normal Darboux
coordinates. Assume that $\cup_{\alpha}U_{\alpha}$ is a tubular
neighbor of $Y$ and $\varphi_{\alpha}\prec U_{\alpha}\cap Y$ is a
partition of unity subordinate to this cover.

In each coordinate chart, let
$\pi_{\alpha}:U_{\alpha}\longrightarrow U_{\alpha}\cap Y$ be the
projection along the symplectic orthogonal. For a function $f$,
define $f_1=\sum_{\alpha}~{\pi_{\alpha}}^*(\phi_{\alpha}f)$ and
extend $f_1$ to $\hat{f}$ on $X$. [Choose an open set $U$ that
separates $Y$ and $X-\cup_{\alpha}U_{\alpha}\subset U$, and further
use partition of unity.]

The process of extension is independent of the function $f$, and
Poisson bracket can be compared in local charts to check the
property required from the extension as Poison bracket in any
Darboux coordinates is given by

\[\{f,g\}=\sum_i \frac{\partial f}{\partial x_i}\frac{\partial g}{\partial \xi_i}-\frac{\partial f}{\partial \xi_i}\frac{\partial g}{\partial x_i}.\]

\end{proof}

We  call an extension $C^{\infty}(Y)\longrightarrow C^{\infty}(X)$
 a symplectic  extension if for all
$f,~g\in C^{\infty}(Y)$ the  identity
$\{\hat{f},~\hat{g}\}_{X_{|Y}}=\{f,~g\}_Y$ holds.

Let $\gp $ be a finite group acting on $X$ by diffeomorphisms, and
let $Y=X^{\gamma}$ be the fixed point manifold for some $\gamma\in \gp$. Then $\gp $ acts on $T^*X$ by symplectomorphisms. We would
identify $T^*Y=T^*(X^{\gamma})$ with $(T^*X)^{\gamma}$.

\begin{cor}
For the inclusion of the symplectic submanifold
$T^*X^{\gamma}\hookrightarrow T^*X$, there exists a symplectic extension which is $\gp_{\gamma}$ equivariant.
\end{cor}


------------------

\section{Canonical Homology of Symplectic
Manifolds}\label{section-CHSM}
Let $X$ be a $2n$-dimensional symplectic manifold. Let
$I:T^*X\rightarrow TX$ be the isomorphism induced by symplectic form
$\omega$, namely, if $\alpha\in T^*X$ and $\xi\in TX$ then $\langle
\alpha,\xi\rangle=\omega(\xi,I(\alpha))$. Then Poisson tensor on $X$
is defined as $G=-\wedge^2I(\omega)$. Let $i(G)$ be an interior
product by $G$, that is to consider a bilinear pairing $G:T^*X\times T^*X\rightarrow \smooth(X)$. For all $k\geq0$, denote by $\wedge ^KG$ the associated bilinear pairing $\wedge^kG:\wedge^k T^*x\times \wedge^k T^*X\rightarrow \smooth(X)$, which is $(-1)^k$ symmetric.
Then $G$ can be used to define the Poisson structure on $X$ instead as $\{f,g\}:=i(G)(df\wedge dg)$.
Set $v_X:=\omega^n/n!$ as the volume form on $X$. The symplectic $*$
operator is the map $*:\Omega^k(X)\rightarrow\Omega^{2n-k}(X)$ which
is defined by
\[\beta\wedge*\alpha=\wedge^kG(\beta,\alpha)v_X.\]

\begin{definition}[Koszul complex]
Let $\delta=i(G)\circ d-d\circ i(G)$. Then
\[\delta:\Omega^k(M)\rightarrow \Omega^{k-1}(M),\]
and is defined in local expression as
\begin{align*}\label{delta}
  \delta(f_0df_1\wedge \ldots &\wedge df_k)
  =\sum_{1\leq i\leq k} (-1)^{i-1}\{f_0,f_i\}_Mdf_1\wedge \ldots \wedge \widehat{df_i} \ldots \wedge df_k\\
&\\
 & +\sum_{1\leq i<j\leq k}(-1)^{i+j-1}f_0d\{f_i,f_j\}_M\wedge df_1 \wedge
  \ldots\widehat{df_i}\wedge \ldots \wedge \widehat{df_j}\wedge \ldots \wedge
  df_k.
\end{align*}

Then $\delta^2=0$ is in fact a differential and we call the
complex$(\Omega^*(M),\delta)$ as the Koszul complex of the Poisson
manifold $(M<\{\})$, and its homology as Poisson homology and denote
it by $HK_*(M)$.
\end{definition}

Although $\delta$ can be defined for any Poisson
manifold, we would restrict to the symplectic case. The following
results were proved in \cite{Brylinski}.
\begin{lem}$X$ be a symplectic manifold. Then $\delta,d$ and $*$ satisfy
the following relations,
\begin{enumerate}
\item If $\alpha\in\Omega^k(X)$ then $**\alpha=\alpha$
\item $\delta=(-1)^{k+1}*d*$ on $\Omega^k(X)$.
\item $d\delta+\delta d=0$ aa a  map on $\Omega^k(X)$.
\end{enumerate}
\end{lem}

A direct consequence of the lemma above is the following.
\begin{prop}[The Poisson Homology of a symplectic Poisson manifold]
Let $(M,\omega)$ be a compact symplectic manifold of dimension $2n$.
Then $(-1)^{k+1}*$ is a chain map between $(\Omega^k(X),\delta)$ and
$ (\Omega^{2n-k}(X),d)$, and the symplectic Poisson homology is
given by
\[HK_k(M)=H^{2n-k}_{deRham}(M).\]
\end{prop}

Now let $X=T^*M/\{0\}$ be the cotangent bundle minus the zero
section. Then there is a $\RR^+$ action on $X$.
Let $\Xi$ be the Euler vector field on $X=T^*M/\{0\}$ generated by
the cation of $\RR^+$. The canonical one form on $X$ is the form
$\alpha=i(\Xi)\omega$. Let
$\varepsilon(\alpha):\Omega^k(X)\rightarrow\Omega^{k+1}(X)$ be the
exterior multiplication by $\alpha$. We would find the following
result useful.
\begin{lem}\label{magic} Let $\maL_Y$ denote the Lie derivative with respect
 to a vector field $Y$. On k-forms $\Omega^k$, the operator
 \[\delta\varepsilon(\alpha)+\varepsilon(\alpha)\delta=\maL_{\Xi}+n-k\]
\end{lem}
In particular, if $\beta$ is a k-form which is homogeneous of degree  $l$, then $\maL_{\Xi}\beta=l\beta$.

\section{Spectral Sequence for the
Cross-Product}\label{section-SSCP}

We come back to $\maB(M)=\maA(M)\rtimes \gp=\pdoM\rtimes \gp$ the cross-product algebra. Since the group action on $\maA(M)$ preserves
the filtration, $\maB(M)$ would again be a filtered algebra
$\maF_i\maB(M)=\maF_i\maA(M)\rtimes \gp$. Let $P$ be a positive elliptic operator invariant under the $\gp$ action. Let $\rho$ be
the symbol for $P$. Then for any $Q\in \maF_m\maA(M)$ an order $m$
operator, the map $\maF_m\maA(M).\maF_{m-1}\maA(M)\ni[Q]\rightarrow
[Q]/\rho^m|_{S^*M}\in \smooth(S^*M)$ identifies the two Frechet spaces. The associated graded algebra would then be

\begin{multline*}
\graded(\maB(M))\simeq \bigoplus_{l\in \ZZ}\dot{\smooth}(T^*M/\{0\})\\
 \bigoplus_{l\in \ZZ}
\left( \smooth(S^*M)\rtimes \gp \right) \rho^l\simeq
\smooth(S^*M)\rtimes \gp\otimes \CC[\rho,\rho^{-1}].
\end{multline*}

To compute the Hochschild homology of $\maB(M)$, we consider the
spectral sequence associated to this filtration.
The $E^1$ term in the Hochschild homology spectral sequence of
$\maB(M)$ would be
\[E^1=HH_*(\graded(\maB(M)))\simeq HH_*((\smooth(S^*M)\rtimes \gp)\otimes
\CC[\rho,\rho^{-1}]) \]\label{HKR2} which, by Theorem \ref{HKR}, can be identified as
$E^1=\oplus_{\langle\gamma\rangle}\dot{\Omega}^*(T^*M/\{0\}^\gamma)^{\gp_{\gamma}}
 =\oplus_{\langle\gamma\rangle}E^1_{\gamma}$. The theorem below describes the $d^1$
 differential under this identification.

\begin{theo}
For every $k\geq 0$, we have the following diagram
\begin{center}
   $ \xymatrix{ {E^1_{k\gamma}=HH_k(\graded(\maB(M))\!)_{\gamma}}     \ar[r]^{\chi_k}     \ar[d]^{d^1}& {\dot{\Omega}^k(T^*M^{\gamma}/\{0\} )}  \ar[d]^{\delta}   \\
     {E^1_{k\!-\!1\gamma}\!=\!HH_{k\!-\!1}\!(\graded(\maB(M))\!)_{\gamma}}
     \ar[r]^{\chi\!_{k\!-\!1}} &{\dot{\Omega}^{k\!-\!1}(T^*M^{\gamma}/\{0\}) } } $
\end{center}
where the differential $\delta$ is given by
\begin{multline*}
  \delta(f_0df_1\wedge \ldots \wedge df_k)=\sum_{1\leq i\leq k}
  (-1)^{i-1}\{f_0,f_i\}_{T^*M^{\gamma}}df_1\wedge \ldots \wedge \hat{df_i} \ldots \wedge df_k\\
  +\sum_{1\leq i<j\leq k}(-1)^{i+j-1}f_0d\{f_i,f_j\}_{T^*M^{\gamma}}\wedge
  df_1 \wedge \ldots\hat{df_i} \wedge \ldots \wedge \hat{df_j}\wedge \ldots
  \wedge df_k.
\end{multline*}
\end{theo}
\begin{proof}
Let $f_i\in\smooth(T^*M^{\gamma}/\{0\})\, , 0\leq i\leq k$ be
homogeneous functions of degree $p_i$ and $\gp_{\gamma}$ invariant.
Let $p=p_0+p_1+\ldots +p_k$. Also, let \[\xi=f_0df_1df_2\ldots
df_k\in \dot{\Omega}^k(T^*M^{\gamma}/\{0\}).\] Then the k-form $\xi$
is defined on the cotangent space minus the 0-section of the fixed
point manifold $M^{\gamma}$ is homogeneous of degree p, and is
$\gp_{\gamma}$ invariant. More importantly, $\xi$ generates
$\dot{\Omega}^k(T^*M^{\gamma}/\{0\})$. By the isomorphism
(\ref{HKR2}), it gives a class in
$HH_k(\smooth(T^*M/\{0\}))_{\gamma}$ represented by
\[E_k(\xi)=\sum_{\pi\in \mathcal{S}_k} \epsilon(\pi)\tilde{f_0}\gamma\otimes
\tilde{f_{\pi(1)}}e\otimes \tilde{f_{\pi(2)}}e\ldots\otimes
\tilde{f_{\pi(n)}}e,\] where $\tilde{f_i}$ is a symplectic extension
of $f_i$. Since $\gp_{\gamma}$ action on $T^*M/\{0\}$ are through
symplectomorphisms, the extensions $\tilde{f_i}$ are in fact
invariant extensions of $f_i$ to whole of $T^*M/\{0\}$.
We only have to check that $\delta = \chi\circ d^1 \circ E$. To
evaluate $d^1$ on $HH_*(\graded(\maB(M)))_{\gamma}$, we must lift
the form $\xi$ to a tensor. We now choose operators $A_0,A_1,\ldots
A_k\in\maF_{p_i}\maA(M)=\Psi^{p_i}(M)/\Psi^{-\infty}(M)$, which are
$\gp_{\gamma}$ invariant and such that $\sigma(A_j)=\tilde{f_j}$.
(This can be done by averaging on a $\gp$ equivariant splitting of
$\sigma$.) Using the above choice of lifting,
 \[\sigma^{-1}(E_k(\xi))=\sum_{\pi\in \mathcal{S}_k} \epsilon(\pi)
 A_0\gamma\otimes A_{\pi(1)}e\otimes A_{\pi(2)}e\ldots\otimes A_{\pi(n)e}.\]
On applying $b$ the Hochschild differential, there are three
different kinds of expressions for each $\pi\in\mathcal{S}_k$. So
the resultant expression can be broken up as the sum of
\[\textbf{A}= \sum_{\pi\in
\mathcal{S}_k}\epsilon(\pi)A_0A_{\pi(1)}\gamma\otimes
A_{\pi(2)}e\otimes \ldots \otimes A_{\pi(k)} e ,\]
\[\textbf{B}= \sum_{\pi\in \mathcal{S}_k}\sum_{i=1}^{k-1}(-1)^{-1}
\epsilon(\pi)A_0\gamma\otimes\ldots\otimes
A_{\pi(i)}A_{\pi(i+1)}e\otimes \ldots \otimes A_{\pi(k)}e ,\]
\[\textbf{C}= \sum_{\pi\in \mathcal{S}_k}(-1)^k\epsilon(\pi)A_{\pi(k)}A_0
\gamma\otimes A_{\pi(1)}e\otimes \ldots \otimes A_{\pi({k-1})}e .\]
(Remember that $A_i$'s have been chosen to be invariant under
$\gamma$.)
By replacing $\pi$ in \textbf{A} by $\pi\tau$ where $\tau$ is a
cyclic permutation, we could rewrite it as
\[\textbf{A}= \sum_{\pi\in
  \mathcal{S}_k}(-1)^{k+1}\epsilon(\pi)A_0A_{\pi(k)}\gamma\otimes
A_{\pi(1)}e\otimes \ldots \otimes A_{\pi({k-1})}e .\] Thus
\[\textbf{A}+\textbf{C}= \sum_{\pi\in \mathcal{S}_k}(-1)^{k+1}
\epsilon(\pi)\{A_0,A_{\pi(k)}\}\gamma\otimes A_{\pi(1)}e\otimes
\ldots \otimes A_{\pi({k-1})}e
\]
or in $HH_*(\graded(\maB(M))$, it is represented by
\[\sum_{\pi\in \mathcal{S}_k}(-1)^{k+1}\epsilon(\pi)\{f_0,f_{\pi(k)}\}\gamma
\otimes f_{\pi(1)}e\otimes \ldots \otimes f_{\pi({k-1})}e.
\]
For any permutation $\pi$ for which $\pi(n)=i$ is fixed, the image
of the tensor
$$(-1)^{k+1}\epsilon(\pi)\{f_0,f_{\pi(k)}\}\gamma\otimes
f_{\pi(1)}e\otimes \ldots \otimes f_{\pi({k-1})}e$$
under $\chi$ of
Theorem \ref{HKR2} is the same namely
\[\frac{1}{k\!-\!1!}(-1)^i\{f_0,f_i\}df_1df_2\ldots \hat{df_i}\ldots df_n.\]
(Again remember that all the $f_j$'s are $\gp_{\gamma}$ invariant.)
There are $(k\!-\!1)!$ permutations such that for a fixed $i$,
$\pi(n)=i$, and therefore the parts corresponding to \textbf{A} and
\textbf{C} in $d^1$ become
\[\sum_{1\leq i\leq k} (-1)^{i-1}\{f_0,f_i\}df_1\wedge \ldots\wedge  df_i \ldots \wedge df_k.\]
The summand \textbf{B} pairs each $\pi$ with the transpositions
$\pi(i\, i\!+\!1)$
\[\textbf{B}= \frac{1}{2}\sum_{\pi\in \mathcal{S}_k}\sum_{i=1}^{k-1}(-1)^{-1}
\epsilon(\pi)A_0\gamma\otimes\ldots\otimes
\{A_{\pi(i)},A_{\pi(i+1)}\}e\otimes \ldots \otimes A_{\pi(k)}e\]
and so $\sigma(\textbf{B})= \frac{1}{2}\sum_{\pi\in
\mathcal{S}_k}\sum_{i=1}^{k-1}(-1)^{i-1}\epsilon(\pi)f_0\gamma\otimes\ldots\otimes
\{f_{\pi(i)},f_{\pi(i+1)}\}e\otimes \ldots \otimes f_{\pi(k)}e.$ All
pairs $(\pi,i)$ such that the set $\{\pi(i),\pi(i-1)\}$ is the same
as the set $\{m,n\},$ $m<n$ have the sane image under $\chi$, namely
the form
\[\frac{(-1)^{m+n}}{k\!-\!1}f_0d\{f_m,f_n\}\wedge df_1\wedge \ldots \wedge
\hat{df_m}\wedge \ldots \wedge \hat{df_n}\wedge \ldots \wedge
df_k.\]
As there would be $2(k-1)!$ such pairs for each $\{m,n\}$, the terms
in \textbf{B} would map to the remainder of $\delta$
\[\sum_{1\leq i<j\leq k}(-1)^{i+j-1}f_0d\{f_i,f_j\} \wedge df_1\wedge
\ldots\hat{df_i}\wedge \ldots \hat{df_j}\wedge \ldots \wedge df_k
.\]
\end{proof}
\begin{cor}\label{HH} We have
\[HH_*(\pdoM\rtimes \gp)=\sum_{\langle \gamma\rangle}H^{2k_{\gamma}-*}
(S^*M^{\gamma}\times S^1)^{\gp_{\gamma}},\] where
$k_{\gamma}=dim(M^{\gamma})$.
\end{cor}
\begin{proof}
Since the $\gp_{\gamma}$ action on $T^*M^{\gamma}/\{0\}$ is by
symplectomorphisms, all $\gamma \in\gp$ commute with the symplectic
$*$ operator and therefore with the operator $\delta$. Hence by
Theorem \ref{HKR} and the previous proposition, $E^2_{k\gamma}\simeq
  H^{2k_{\gamma}-*}(S^*M^{\gamma}\times S^1)^{\gp_{\gamma}}$.

We now observe that the differential $d^2$ for this spectral
sequence vanishes. The $E^1$ can be given a $\ZZ$ bigrading with $E^1_{k,l}$ be
k-forms of homogeneity $l$ on $T^*(M)\{0\}$. Then the differential
$d^1=\delta$ maps $E^1_{lk}\rightarrow E^1_{k-1l-1}$. This is
because the Poisson bracket decreases the homogeneity by $1$. But on
$E^1_{kl}$ by \ref{magic} we have
\[\delta\varepsilon(\alpha)+\varepsilon(\alpha)\delta=\maL_{\Xi}+n-k=l+n-k.\]
Thus unless $l=k-n$ there is a contraction for $\delta$

on$E^1_{pq}$. Therefore the only nonzero terms on $E^2_{kl}$

correspond to $l=n-k$. The differential $d^2$ must either start or
end in a $0$ term.
\end{proof}

Since the Hochschild homology groups are finite dimensional, the
Hochschild cohomology groups are the dual of the Hochschild homology
groups. We still write down explicitly $HH^0(\maB(M))$, the space of
traces on $\maB(M)$
\begin{cor}
We have $HH^0(\maB(M))=\sum_{\langle \gamma
\rangle}H_0(S^*M^{\gamma}\times S^1)$. hence its rank is the number
of path components of $S^*M^{\gamma}$.
\end{cor}

\begin{proof}
We have $HH_0(\maB(M))_{\gamma} = H^{2k_{\gamma}}(S^*M^{\gamma}
\times S^�)$, because the action of $\gp$ on $S^*M$ is orientation
preserving. By Poincare duality, $
H^{2k_{\gamma}}(S^*M^{\gamma}\times S^1)= H^{0}(S^*M^{\gamma}\times
S^1).$
\end{proof}

To compute the cyclic homology, we use Connes' $SBI$ exact sequence:
\[\ldots\stackrel{B}{\to} 
HH_n(\maB(M)) \stackrel{I}{\to} HC_n(\maB(M))\stackrel{S}{\to}
HC_{n-2}(\maB(M))\stackrel{B}{\to} HH_{n-1}(\maB(M))\to\ldots.\] Here
the connecting morphism $B:Tot(B_*(\maB(M)))\rightarrow
\maH_*(\maB(M))[-1]$ is the cyclic boundary map $B$.
\begin{prop}
The connecting morphism $B$ in the $SBI$ exact sequence for
$\maB(M)$ vanishes and hence,
\[HC_K(\maB(M)(M)) = \sum_{j\ge 0}
HH_{k-2j}(\maB(M)(M)).\]
\end{prop}
\begin{proof}
Since $\maB(M)$ is unital, $B$ is induced from a chain map
\[Tot(B_*(\maB(M)))\rightarrow \maH_*(\maB(M))[-1]\] that respects the
filtration on both the source and the range complex. Thus $B$ induces
a map on the spectral sequence $EC^*$ of $Tot(B_*(\maB))$ into
spectral sequence $EH^*$ of $\maH_*(\maB(M))[-1]$. The induced map
on $EC^1\rightarrow EH^1$ when identified with the space of
differential forms is the   deRham differential $d$ which vanishes on
$E^2$.
$$d:EC^1_{k,l}\oplus_j\Omega^{k-2j}(S^*M^{\gamma}\times
S^1)^{|gp_{\gamma}}_l\rightarrow
EH^1_{k,l}\Omega^{k+1}(S^*M^{\gamma}\times S^1)^{\gp_{\gamma}}_l.$$

For the $EC$ term to contribute to nonzero homology,
$l=2k_{\gamma}-k$, where as for the $EH$ term $l=2k_{\gamma}-k-1$.

The formula for $HC_K(\maB)$ then follows from the $SBI$ sequence.
\end{proof}

We can note that the map induced by $B$ on $EH^1$ is a differential
implies in particular that
\begin{cor}$HP_j(\maB(M))=\sum_{k\in\ZZ}HH_{j-2k}(\maB(M))$.
\end{cor}

\end{document}